\documentclass{amsart}

\pdfoutput=1

\usepackage[hmarginratio={1:1}]{geometry}
\usepackage[shortlabels]{enumitem}
\usepackage{mathrsfs}
\usepackage[hidelinks]{hyperref}
\usepackage[style=alphabetic,
doi=false,
isbn=false,
maxnames=99,
maxalphanames=99,
backref=false]{biblatex}

\usepackage{soulutf8}
\usepackage{amssymb}
\usepackage{tikz-cd}

\newcommand\AAND{\quad\text{and}\quad}

\newcommand\defeq{\mathrel{\vcenter{\baselineskip0.5ex \lineskiplimit0pt \hbox{.}\hbox{.}}} =}
\newcommand\defword[1]{{\boldmath \textbf{#1}}}
\newcommand\FV{\mathrm{FV}}
\newcommand\ol[1]{\overline{#1}}

\newcommand\pbig[1]{\big({#1}\big)}
\renewcommand\phi{\varphi}

\newcommand\Rel[1][R]{\sim_{#1}}
\newcommand\set[2][]{#1\{ #2 #1\}}

\newcommand\Tmcl{\ol{\Tm}}
\newcommand\Tm{\ensuremath{\Lambda}}
\newcommand\too{\twoheadrightarrow}
\newcommand\To{\Rightarrow}
\newcommand\unsim{{\sim}}
\newcommand\wt[1]{\widetilde{#1}}
\newcommand\Y{\mathrm{Y}}

\theoremstyle{definition}
\newtheorem{defn}{Definition}

\theoremstyle{remark}

\theoremstyle{plain}
\newtheorem{lem}[defn]{Lemma}
\newtheorem{propn}[defn]{Proposition}
\newtheorem*{thm}{Theorem}

\let\oldproofname=\proofname
\renewcommand{\proofname}{\rm\bf{\oldproofname}}

\title{Y is a least fixed point combinator}
\author{Joseph Helfer}
\date{\today}

\begin{document}
\maketitle

\vspace{-0.5cm}

\begin{abstract}
  The theory of recursive functions is related in a well-known way to the notion of \emph{least fixed points}, by endowing a set of partial functions with an ordering in terms of their domain of definition.
  When terms in the pure \( \lambda \)-calculus are considered as partial functions on the set of reduced \( \lambda \)-terms, they inherit such a partial order.
  We prove that Curry's well-known fixed point combinator \( \Y \) produces least fixed points with respect to this partial order.
\end{abstract}

\section{Introduction}

Curry's ``paradoxical combinator'' \( \Y \) in the pure lambda calculus, defined as
\begin{equation}\label{eq:Y-defn}
  \Y = \lambda f. \pbig{\lambda g. f(g g)}\pbig{\lambda g. f(g g)},
\end{equation}
is such that \( \Y F \) is a fixed point of \( F \) for any lambda term \( F \), in the sense that \( F(\Y F) \) and \( \Y F \) are \( \beta \)-equivalent.
\( \Y \) is sometimes called the ``least fixed point operator'' because there are models of the pure lambda calculus in which the domain is equipped with a partial order, and in which it the interpretation of \( \Y \) is a function taking each element of the domain to a least fixed point with respect to that partial order (see, e.g., \cite[Theorem~19.3.4]{barendregt-lambda-calculus}).

However, the set of lambda terms itself can be equipped with the usual partial order (or rather, preorder) on partial functions by declaring that \( F \le G \) for lambda terms \( F \) and \( G \) if and only if, for all lambda terms \( M \), if \( F M \) has a \( \beta \)-normal form, then \( G X \) has the same \( \beta \)-normal form (Definition~\ref{defn:the-order}).
One can then ask whether \( \Y F \) is always minimal with respect to this order among fixed points of \( F \).
Surprisingly, this questions does not seem to have appeared in the literature.
In this note, we answer it in the affirmative.

In \S\ref{sec:others}, we briefly discuss the question of other fixed-point combinators.
\\[-0.8em]

\noindent
\textbf{Acknowledgements:}
We thank Henk Barendregt for helpful discussions concerning this paper.

\section{Preliminaries}
Let \( V \) be an infinite set (of ``variables'').
We write \( \Tm \) for the set of \defword{lambda terms} with variables in \( V \), up to \( \alpha \)-equivalence.
We will refer to the elements of \( \Tm \) simply as ``lambda terms'' (or just ``terms'') rather than ``lambda terms up to \( \alpha \)-equivalence''.

Rather than recalling the explicit definition of \( \Tm \) (for which see, e.g., \cite[Chapter 2]{barendregt-lambda-calculus}), we state certain properties which uniquely characterize it.
More precisely, we will state certain properties of (a) \( \Tm \), together with its operations \( V \to \Tm \) (atomic terms), \( \Tm \times \Tm \to \Tm \) (application), and \( V \times \Tm \to \Tm \) (abstraction), (b) the set \( \FV(M) \subset V \) of \defword{free variables} of a lambda term \( M \), and (c) the operation of \defword{substitution} \( M[v \defeq N] \) of lambda terms (\( M \) and \( N \) terms and \( v \) a variable).
These properties determine \( \Tm \) together with the mentioned operations on it uniquely up to isomorphism, and determine the sets \( \FV(t) \) and the substitution operation uniquely:
\begin{enumerate}[(a)]
\item Every term is of exactly one of the forms (i) \( v \) with \( v \in V \), (ii) \( M N \) for \( M, N \in \Tm \), or (iii) \( \lambda v. M \) with \( v \in V \) and \( M \in \Tm \).
In case (i), \( v \) is uniquely determined, and in case (ii), \( M \) and \( N \) are uniquely determined.
For case (iii), if \( \lambda u. M = \lambda v. N \), then \( N = M[u \defeq v] \).
Moreover \( \Tm \) is the least subset of \( \Tm \) closed under the operations (i)-(iii), i.e., we can use structural induction and recursion on \( \Tm \).
\item The set \( \FV(M) \subset V \) of free variables of a term \( M \) is determined by (i) \( \FV(v) = \set{v} \) for \( v \in V \), (ii) \( \FV(M N) = \FV(M) \cup \FV(N) \), and (iii) \( \FV(\lambda u. M) = \FV(M) - \set{u} \).
\item The substitution operation \( M[v \defeq N] \) is determined by (i) \( v[v \defeq M] = M \) and \( u[v \defeq M] = u \) for \( u \in V - \set{v} \), (ii) \( (M_1 M_2)[v \defeq N] = (M_1[v \defeq N])(M_2[v \defeq N]) \), and (iii) \( (\lambda u.M_1)[v \defeq s] = \lambda u.(M_1[v \defeq s]) \) for \( u \in V - (\FV(s) \cup \set{v}) \).

Note that if \( u \in \FV(s) \cup \set{v} \), we can always take some \( u' \in V - (\FV(s) \cup \set{v}) \), and set \( M_1' = M_1[u \defeq u'] \), and we then have \( \lambda u. M_1 = \lambda u'. M_1' \), so the clauses (i)-(iii) indeed suffice to determine the substitution operation.
\end{enumerate}

For a relation \( R \subset \Tm \times \Tm \), we write \( M \sim_R N \) for \( (M,N) \in R \).
We say that \( R \) is \defword{compatible} (see \cite[3.1.1]{barendregt-lambda-calculus}) if given terms \( M,M',N,N' \) with \( M \Rel M' \) and \( N \Rel N' \), then \( M N \Rel M' N \), \( M N \Rel M N' \), and \( \lambda v.N \Rel \lambda v.N' \) for \( v \in V \).
For a relation \( R \), we write \( \to_R \) (\defword{one-step \( R \)-reduction})) for the least compatible relation containing \( R \), we write \( \too_R \) (\defword{\( R \)-reduction}) for the least preorder containing \( \to_R \), and \( \approx_R \) (\defword{\( R \)-equivalence}) for the least equivalence relation containing \( \too_R \) (or equivalently, containing \( \to_R \)).
It follows that \( \too_R \) and \( \approx_R \) are themselves compatible, and are hence the least compatible preorder and equivalence relation containing \( R \), respectively.

The relations \( \beta,\eta \subset \Tm \times \Tm \) are defined as
\[
  \beta = \set{\pbig{(\lambda v. M) N, M[v \defeq N]} \mid v \in V;\ M,N \in \Tm}
  \AAND
  \eta = \set{\pbig{(\lambda v. M v), M} \mid v \in V;\ M \in \Tm},
\]
and we define \( \beta\eta = \beta \cup \eta \).
We thus obtain the relations \( \to_\beta \), \( \too_\beta \), and \( \approx_\beta \) of one-step \( \beta \)-reduction, \( \beta \)-reduction, and \( \beta \)-equivalence, and likewise \( \to_{\beta \eta} \), \( \too_{\beta \eta} \), and \( \approx_{\beta \eta} \).

\begin{lem}\label{lem:compat-subst}
  If \( R \) is either of \( \beta \) or \( \beta \eta \) (or more generally if \( \beta \subset R \)) then
  \[
    M \approx_R M' \To M[v \defeq N] \approx_R M'[v \defeq N]
  \]
  for all \( M,M',N \in \Tm \) and \( v \in V \).
\end{lem}
\begin{proof}
  Assuming \( M \approx_R M' \), we have \( M[v \defeq N] \approx_R (\lambda v. M) N \approx_R (\lambda v. M') N \approx_R M'[v \defeq N] \).
\end{proof}

For the remainder of this section, fix a relation \( R \subset \Tm \times \Tm \).

A term \( M \) is in \defword{\( R \)-reduced}, if there is no term \( N \) with \( M \to_R N \).
An \defword{\( R \)-normal form} of a term \( M \) is an \( R \)-reduced term \( N \) with \( M \approx_R N \).
We say that \( R \) has the \defword{Church-Rosser property} if for all \( M,N,N' \in \Tm \), if \( M \too_R N \) and \( M \too_R N' \), then there is \( L \in \Tm \) with \( N \too_R L \) and \( N' \too_R L \).
This implies that every term has at most one \( R \)-normal form, and that if \( N \) is the \( R \)-normal form of \( M \), then \( M \too_R N \).
By the Church-Rosser theorem, \( \beta \) and \( \beta \eta \) both have the Church-Rosser property \cite[\S11.1]{barendregt-lambda-calculus}.

The \defword{Y-combinator} is the lambda term \( \Y \) defined in \eqref{eq:Y-defn}.

\begin{defn}\label{defn:the-order}
  We define a preorder \( \le_R \) on terms by putting, for \( F,G \in \Tm \):
  \[
    F \leq_R G \iff ( \forall M \in \Tmcl.\ F M \text{ has an \( R \)-normal form} \To G M \text{ has the same \( R \)-normal form} ).
  \]
\end{defn}

\section{The proof}
\begin{thm}($\Y$ yields $R$-minimal fixed points.)
  Let \( R \) be either \( \beta \) or \( \beta \eta \). Then
\[\forall F,M \in \Tmcl.\ FM \approx_R M \Rightarrow \Y F \leq_R M.\]
\end{thm}

We first give an outline of the proof.
Given a fixed point \( M \) of \( F \), and a term \( N \) such that \( (\Y F) N \) has a normal form \( N' \), we must show that \( M N \) has the same normal form.
The idea is that the reduction \( (\Y F) N \too_R N' \) should use ``nothing about'' \( \Y F \) \emph{except} that it is a fixed point \( F \).
Thus, we introduce a new variable \( y \) which we consider as a ``formal fixed point of \( y \)'' by introducing a new reduction rule \( y \to F y \), and show (i) that the reduction \( F N \too_R N' \) gives rise to a parallel reduction \( y N \too N' \), and (ii) that the reduction \( y N \too N' \) gives rise to a parallel reduction \( M N \to N' \).
This second step is easy, by substituting \( M \) for \( y \) in the reduction \( y N \too N' \).

Step (i) is a little trickier, since in the reduction \( (\Y F) N \too N' \), we can make reductions \emph{inside of} \( \Y F \), whereas we cannot make reductions ``inside of \( y \)''.
This is taken care of by allowing each instance of \( y \) appearing in the reduction \( y N \too N' \) to ``correspond'' to an arbitrary \emph{reduction} of \( \Y F \), and keeping track of which instances of \( y \) correspond to which such reductions.

We now proceed with the proof.
Henceforth, let \( R \) be one of \( \beta \) or \( \beta \eta \).

\begin{defn}
  For any \( y \in V \) and \( F \in \Tm \), we define the relation \( S_{F,y} \subset \Tm \times \Tm \) by \( S_{F,y} = \set{(y, F y)} \), and we set \( R_{F,y} = R \cup S_{F,y} \).
  We thus have the relations \( \to_{S_{F,y}} \), \( \too_{S_{F,y}} \), and \( \approx_{S_{F,y}} \) of one-step \( S_{f,z} \)-reduction, \( S_{F,z} \)-reduction, and \( S_{F,z} \)-equivalence.
\end{defn}

\begin{lem}\label{lem:sz-g-sub}
  Let \( F \in \Tm \) and \( y \in V - \FV(F) \). Then
  \[
    F M \approx_R M \wedge
    N \approx_{R_{F,y}} N' \To
    N[y \defeq M] \approx_R N'[y \defeq M]
  \]
  for any \( M,N,N' \in \Tm \).
\end{lem}

\begin{proof}
  Since \( \unsim_{R_{F,y}} \) is the least compatible equivalence relation containing \( R_{F,y} \), and since \( \set{(N, N') \mid N[y \defeq M] \approx_R N'[y \defeq M]} \) is itself a compatible equivalence relation, it suffices to prove the conclusion in the case where \( N \Rel[R_{F,y}] N' \).

  If \( N \Rel N' \) (or more generally if \( N \approx_R N' \)), then the conclusion follows immediately from Lemma~\ref{lem:compat-subst}.
  If \( N \Rel[S_{F,y}] N' \), then \( N = y \) and \( N' = F y \), hence \( N [y \defeq M] = M \approx_R F M = N'[y \defeq M] \) since \( y \notin \FV(F) \).
\end{proof}

\begin{lem}\label{lem:reduced-no-fv}
  Let \( F \in \Tm \) and \( z \in V \).
  If \( M \in \Tm \) is \( R_{F,y} \)-reduced, then \( y \notin \FV(M) \).
\end{lem}

\begin{proof}
  By induction on \( M \), using that \( y \) is not \( R_{F,y} \)-reduced.
\end{proof}

We now dispense with ``step (ii)'' from the proof outline above.

\begin{propn}\label{propn:step-1}
  Let \( F,M,N \in \Tm \) and suppose \( F M \approx_R M \).
  For any \( y \in V - \FV(F) \cup \FV(M) \cup \FV(N) \), if \( y N \) has an \( R_{F,y} \)-normal form, then \( M N \) has the same \( R \)-normal form.
\end{propn}
\begin{proof}
  If \( y N \) has an \( R_{F,y} \)-normal form, then there is a sequence \( y N = N_0, \ldots, N_n \) where \( N_n \) is \( R_{F,y} \)-reduced and \( N_i \to_{R_{F,y}} N_{i+1} \) for each \( i < n \).
  Now consider the sequence \( N_0[y \defeq M],\ldots,N_n[y \defeq M] \).
  We have \( N_0[y \defeq M] = M N \), and by Lemma~\ref{lem:reduced-no-fv}, \( N_n[y \defeq s] = N_n \), which is \( R_{F,y} \)-reduced and hence \( R \)-reduced.
  By Lemma~\ref{lem:sz-g-sub}, \( N_i[y \defeq M] \approx_R N_{i+1}[y \defeq M] \) for all \( i \), and hence \( M N \approx_R N_n \), as desired.
\end{proof}

We now proceed to ``step (i)'' in the above proof outline.
We begin with some preliminary definitions.
As indicated in the outline, we will need to get a handle on what arbitrary \( R \)-reductions of \( \Y F \) look like.

\begin{defn}\label{defn:upsilon-f}
  Given \( M,N \in \Tm \), let \( \Y_{M,N} \) be the term \( \pbig{\lambda g. M (g g)} \pbig{\lambda g. N (g g)} \), where \( g \notin \FV(M) \cup \FV(N) \), so that \( \Y = \lambda f. \Y_{f,f} \).

  For \( n \ge 0 \) and \( M,N \in \Tm \), we define \( M^{(n)} N \) recursively by \( M^{(0)}N = N \) and \( M^{(n+1)} N = M (M^{(n)} N) \).
  For \( n \ge 0 \), define \( \Y_n \) to be the term \( \lambda f. f^{(n)} \Y_{f,f} \), so that \( \Y_0 = \Y \).

  Finally, for \( F \in \Tm \), we let
  \[
    \Upsilon_F =
    \set{ \Y_{n} F' \mid F \too_R F', n \ge 0 }
    \cup
    \set{ \Y_{F',F''} \mid F \too_R F', F \too_R F'' }
    \subset \Tm.
  \]
\end{defn}

\begin{lem}\label{lem:upsilon-f-reduce}
  Fix \( F \in \Tm \).

  Let \( M \in \Upsilon_F \) and suppose \( M \to_R M' \) for some \( M' \in \Tm \).

  Then \( M' = (F')^{(n)} N \) for some \( F' \in \Tm \) with \( F \too_R F' \), some \( N \in \Upsilon_F \), and some \( n \ge 0 \).
\end{lem}

\begin{proof}
  Note first that the unique one-step \( R \)-reduction of \( \Y_{n} \) for \( n \ge 0 \) is \( \Y_{n+1} \), as is proven by induction on \( n \).
  Next, for any \( K, L \in \Tm \), the only one-step \( R \)-reductions of \( \Y_{K,L} \) are \( K(\Y_{L,L}) \) and \( \Y_{K',L} \) or \( \Y_{K,L'} \) with \( K \to_R K' \) and \( L \to_R L' \).

  Now, let \( M \) and \( M' \) be as in the hypothesis.
  If \( M = \Y_{n} F' \) with \( F \too_R F' \), then the three possibilities for \( M' \) are \( \Y_{n+1} F' \in \Upsilon_F \), \( \Y_{n} F'' \in \Upsilon_F \) where \( F' \to_R F'' \), or finally \( (F')^{(n)} \Y_{F',F'} \), where \( Y_{F',F'} \in \Upsilon_F \).
  If \( M = \Y_{F',F''} \), then \( M' \) is either \( F' \Y_{F'',F''} \) or is of the form \( \Y_{F',F'''} \) or \( \Y_{F''', F''} \) with \( F \too_R F''' \).
\end{proof}

It will be convenient for us to consider an enlarged set \( \wt V = V \cup V' \) of variables, where \( V' \) is some set disjoint from \( V \).
In fact, we take \( V' = V'_F \) to be (isomorphic to) to the set \( \Upsilon_F \) of Definition~\ref{defn:upsilon-f} for some \( F \in \Tm \).
Given \( M \in \Upsilon_F \), we write \( v_M \) for the corresponding element of \( V'_F \).

The variables in \( V'_F \) will allow us to keep track of ``which variables correspond to which instances of (reductions of) \( \Y f \)'' as indicated in the above outline.

\begin{defn}
  Fix \( F \in \Tm \).

  We write \( \wt \Tm_F \) for the set of lambda-terms with variables in \( \wt V_F = V \cup V'_F \).
For \( M \in \wt \Tm_F \), we write \( \wt{\FV}_F(M) \subset \wt V_F \) for the set of free variables, and we set \( \FV(M) = \wt{\FV}_F(M) \cap V \) and \( \FV'_F(M) = \wt{\FV}(M) \cap V'_F \).
Note that \( \Tm = \set{M \in \wt \Tm_F \mid \FV'_F(M) = \emptyset} \subset \wt \Tm_F \).

We define the \emph{realization} map \( \rho \colon \wt \Tm_F \to \Tm \) by substituting \( M \) for each \( v_M \in V'_F \); and given \( y \in V \), we define the \emph{forgetful} or \emph{flattening} map \( \phi_y \colon \wt \Tm_F \to \Tm \) by substituting \( y \) for each variable in \( V'_F \); both \( \rho \) and \( \phi_{y} \) are defined by recursion in an evident manner.
\end{defn}

\begin{lem}\label{lem:subst-subst}
  Fix \( F \in \Tm \), and let \( M, N \in \wt \Tm_F \).
  Then
  \[ \rho(M[v \defeq N]) = \rho(M) [v \defeq \rho(N)] \]
  for any \( v \in V - \FV(F) \), and
  \[ \phi_y(M[v \defeq N]) = \phi_y(M) [v \defeq \phi_y(N)] \]
  for any \( v \in V \) and \( y \in V - \set{v} \).
\end{lem}
\begin{proof}
  By induction on \( M \), using that \( \rho(K) = \phi_y(K) = K \) for \( K \in \Tm \), and that \( K[w \defeq L] = K \) for \( K,L \in \Lambda \) and \( w \in V - \FV(M) \).
  We note that the hypothesis \( v \notin \FV(F) \) is equivalent to \( v \notin \FV(K) \) for all \( K \in \Upsilon_F \), and it is the latter which is actually used in the proof.
\end{proof}

We have the following variant of Lemma~\ref{lem:reduced-no-fv}:
\begin{lem}\label{lem:rho-reduced-no-upsilon-fv}
  Fix \( F \in \Tm \).
  For any \( M \in \wt \Tm_F \), if \( \rho(M) \) is \( R \)-reduced, then \( \FV'_F(M) = \emptyset \) (i.e., \( M = \rho(M) \in \Tm \)).
\end{lem}

\begin{proof}
  By induction on \( M \), using that no element of \( \Upsilon_F \) is \( R \)-reduced.
\end{proof}

Now comes the crucial definition.
As explained in the above outline, we will be considering a sequence of reductions of a term \( \Y F \), and producing a parallel sequence in which each instance of \( \Y F \) or some reduction of it is replaced by some variable \( y \).
The following definition is what lets us lift each step of the first sequence to a step in the second.

\begin{defn}\label{defn:main-constuction}
  Fix \( F \in \Tm \).

  Given \( N,N' \in \Tm \) and \( M \in \wt \Tm_F \) with \( N \to_R N' \) and \( \rho(M) = N \), we define a new term
  \[
    M' = \gamma_{N,N'}(M) \in \wt \Tm_F
  \]
with \( \rho(M') = N' \) and \( \phi_y(M) \too_{R_{F,y}} \phi_y(M') \) for any \( y \in V - \FV(N) = V - \FV(N') \).
  \[
    \begin{tikzcd}[row sep=10pt]
      N \ar[d, "R" near end] & M \ar[l, mapsto, "\rho"'] \ar[r, mapsto, "\phi_y"] & \phi_y M \ar[d, ->>, "R_{F,y}" near end, shorten <=-3pt] \\
      N' & M' \ar[l, mapsto, "\rho"'] \ar[r, mapsto, "\phi_y"] & \phi_y M'
    \end{tikzcd}
  \]

  The definition of \( M' = \gamma_{N,N'}(M) \) is by recursion on \( N \):
  \begin{itemize}[leftmargin=25pt]
  \item If \( N = v \in V \), then \( N' = M = v \), and we set \( M' = v \); then evidently \( \rho(M') = N' \) and \( \phi_y(M) = \phi_y(M') \too_{R_{F,y}} \phi_y(M') \).
  \item If \( N = N_1 N_2 \), we consider several sub-cases:
    \begin{enumerate}[(i)]
    \item If \( M \in V'_F \), then \( N \in \Upsilon_F \), and by Lemma~\ref{lem:upsilon-f-reduce}, \( N' = (F')^{(n)} N_3 \) with \( N_3 \in \Upsilon_F \) and \( F \too_R F' \).
      We then set \( M' = (F')^{(n)} v_{N_3} \), and then have \( \rho(M') = N' \) and \( \phi_y(M) = y \too_{S_{F,y}} F^{(n)} y \too_R (F')^{(n)} y = \phi_y(M') \).
    \item Otherwise, we have \( M = M_1 M_2 \) with \( \rho(M_i) = N_i \) for \( i = 1,2 \).
      \begin{enumerate}[(\roman{enumi}-a)]
      \item If \( N_1 = \lambda v. L \) with \( v \notin \set{y} \cup \FV(F) \) and \( N' = L[v \defeq N_2] \), then note that we cannot have \( M_1 \in V'_F \), since \( \Upsilon_F \) contains no \( \lambda \)-abstraction terms.
        Thus, we must have \( M_1 = \lambda v. K \) with \( \rho(K) = L \), and we set \( M' = L[v \defeq M_2] \).
        We then have \( \rho(M') = \rho(L)[v \defeq \rho(M_2)] = K[v \defeq N_2] \) by Lemma~\ref{lem:subst-subst}, and we have \( \phi_y(M) = \pbig{\lambda v. \phi_y(L)}\pbig{\phi_y(M_2)} \) and \( \phi_y(M') = \phi_y(L)[v \defeq \phi_y(M_2)] \) again by Lemma~\ref{lem:subst-subst}, and hence \( \phi_y(M) \to_\beta \phi_y(M') \).
      \item Otherwise, we have \( N' = N_1' N_2' \) with \( N_i \to_R N_i' \) and \( N_j = N_j' \) where \( \set{i,j} = \set{1,2} \).
        We then set \( M_i' = \gamma_{N_i,N_i'}(M_i) \) and \( M_j' = M_j \), and set \( M' = M_1' M_2' \).
        We then have \( \rho(M') = \rho(M_1') \rho(M_2') = N_1' N_2' = N' \); and we have \( \phi_M(M_i) \too_{R_{F,y}} \phi_y(M_i') \) for \( i = 1,2 \) and hence \( \phi_y(M) \too_{R_{F,y}} \phi_y(M') \).
      \end{enumerate}
    \item If \( N = \lambda v. N_1 \) with \( v \notin \set{y} \cup \FV(F) \), then \( M = \lambda v. M_1 \) where \( \rho(M_1) = N_1 \).
      We again consider sub-cases:
      \begin{enumerate}[(\roman{enumi}-a)]
      \item If \( N_1 = N_2 v \) for some term \( N_2 \), and \( N' = N_2 \) (i.e., if \( R = \beta \eta \) and \( N_1 \to_{\eta} N_2 \)), then \( M_1 = M_2 v \) where \( \rho(M_2) = N_2 \), and we set \( M' = M_2 \).
        We then have \( \rho(M') = N' \) and \( \phi_y(M) = \lambda v. \phi_y(M_2) v \to_\eta \phi_y(M') \).
      \item Otherwise, \( N' = \lambda v. N_1' \) with \( N_1 \to_R N_1' \).
        We then set \( M_1' = \gamma_{N_1,N_1'}(M_1) \) and set \( M' = \lambda v. M_1' \), and we have \( \rho(M') = \lambda v. \rho(M_1') = \lambda v. N_1' = N' \) and \( \phi_y(M) = \lambda v. \phi_y(M_1) \too_{R_{F,y}} \lambda v. \phi_y(M_1') = \phi_y(M') \).
      \end{enumerate}
    \end{enumerate}
  \end{itemize}
  (\emph{End of definition of \( \gamma_{N,N'} \).})
\end{defn}

The following proposition (applied to \( L = y N \)), together with Proposition~\ref{propn:step-1}, immediately implies the theorem.

\begin{propn}\label{propn:main-propn}
  Let \( F,L \in \Tm \) and \( y \in V - \FV(F) \).
  If \( L[z \defeq \Y F] \) has an \( R \)-normal form, then \( L \) has the same \( R_{F,y} \)-normal form.
\end{propn}

\begin{proof}
Let \( L \) be a term such that \( N_0 = L[y \defeq \Y F] \) has an \( R \)-normal form.
We thus have a sequence \( N_0, \ldots, N_n \) where \( N_n \) is \( R \)-reduced and \( N_i \to_R N_{i+1} \) for all \( 0 \le i < n \).

Define \( M_0 = L[y \defeq v_{\Y F}] \in \wt \Tm_F \), so that \( \rho(M_0) = N_0 \) and \( \phi_y(M_0) = L \).
Now for each \( 0 \le i < n \), define \( M_{i+1} = \gamma_{N_i,N_{i+1}}(M_i) \); we then have \( \rho(M_{i+1}) = N_{i+1} \) and \( \phi_y(M_i) \too_{R_{F,y}} \phi_y(M_{i+1}) \), and in particular, \( L \too_{R_{F,y}} \phi_y(M_n) \) (see Figure~\ref{fig:si-ti-behaviour}).

\begin{figure}
  \[
    \begin{tikzcd}[
      column sep=15pt,
      row sep = 20pt,
      /tikz/column 1/.append style={anchor=base east},
      /tikz/column 3/.append style={anchor=base west}]
      N_0 \ar[r, "R" {anchor=north west, xshift=4pt}, shorten=3pt]
      & \ \cdots \ar[r, "R" {anchor=north west, xshift=4pt}, shorten=3pt] &
      \ \ \ N_n \\
      M_0,  & \cdots, \ar[u, "\rho"', shorten=5pt] \ar[d, "\phi_z", shorten=5pt] &
      \ \ \ M_n \\
      L = \phi_z(M_0)
      \ar[r, ->>, "R_{F,y}" {anchor=north west, xshift=4pt}, shorten=2pt] &
      \ \  \cdots
      \ar[r, ->>, "R_{F,y}" {anchor=north west, xshift=4pt}, shorten=2pt] &
      \ \ \ \phi_y(M_n)
    \end{tikzcd}
  \]
  \caption{The behaviour of the \( M_i \)'s and \( N_i \)'s in Proposition~\ref{propn:main-propn}.}
  \label{fig:si-ti-behaviour}
\end{figure}

Since \( N_n \) is \( R \)-reduced, \( \FV_F'(M_n) = \emptyset \) by Lemma~\ref{lem:rho-reduced-no-upsilon-fv}, hence \( N_n = M_n = \phi_y(M_n) \) and thus \( L \approx_{R_{F,y}} N_n \).
Since \( N_n \) is \( R \)-reduced and \( y \notin \FV(N_n) \) (this follows from \( N_n \approx_R N_0 = L[y \defeq \Y F] \) and \( y \notin \FV(\Y F) \)), it is also \( R_{F,y} \)-reduced, as desired.
\end{proof}

\section{Other fixed point combinators}\label{sec:others}
There are (infinitely many) other terms \( M \) in the \( \lambda \)-calculus which are fixed-point combinators in the sense that \( F (T F) \approx_R T F \) for all \( F \in \Tm \).
A well-known example is Turing's combinator \( \Theta = \pbig{\lambda x. \lambda y. y (x x y)}\pbig{\lambda x. \lambda y. y (x x y)} \); see \cite[Definition~6.1.4]{barendregt-lambda-calculus}.

In the case of \( \Theta \), it is easy to adapt the above proof to see that it is also a least fixed point combinator.
The main point is to modify Definition~\ref{defn:upsilon-f} in light of the possible \( \beta \)-reductions of \( \Theta F \): the set \( \Upsilon_F \) should now consist of all terms \( \Theta' F' \) with \( \Theta \too_R \Theta' \) and \( F \too_R F' \).
The statement of Lemma~\ref{lem:upsilon-f-reduce} then still holds, the construction in Definition~\ref{defn:main-constuction} goes through verbatim, and the statement and proof of Proposition~\ref{propn:main-propn} go through \emph{mutatis mutandis}, substituting \( \Theta \) for \( \Y \).

The proof could similarly be adapted for other fixed-point combinators; it is just a matter of modifying Definition~\ref{defn:upsilon-f} so that Lemma~\ref{lem:upsilon-f-reduce} remains true.
In fact, in light of \cite[Theorem~19.3.4]{barendregt-lambda-calculus}, it is plausible that \emph{every} fixed point combinator is a least fixed point combinator.
Moreover, given the result of \cite{goldberg-enumerability-of-fp-combinators} giving a recursive enumeration of all fixed point combinators, proving this conjecture is perhaps not out of reach.

\printbibliography
\end{document}